\documentclass[12pt]{amsart}

\usepackage{amsthm}
\usepackage{amssymb}
\usepackage{amsmath}
\usepackage[margin=1in]{geometry}
\usepackage{comment}
\usepackage{color}

\newtheorem{thm}{Theorem}[section]
\newtheorem*{main_thm}{Theorem~\ref{thm:main}}
\newtheorem*{main_prop}{Proposition~\ref{prop:main}}
\newtheorem{prop}[thm]{Proposition}
\newtheorem{conj}[thm]{Conjecture}
\newtheorem{lemma}[thm]{Lemma}

\newtheorem{defn}[thm]{Definition}

\theoremstyle{remark}
\newtheorem*{rem}{Remark}

\DeclareMathOperator{\ord}{ord}
\DeclareMathOperator{\Res}{Res}
\DeclareMathOperator{\PGL}{PGL}

\newcommand{\pberk}{\mathbf{P}^1_K}
\newcommand{\hberk}{\mathbf{H}^1_K}
\newcommand{\aberk}{\mathbf{A}^1_K}
\newcommand{\zetaG}{\zeta_G}

\newcommand{\GammaF}{\Gamma_{\mathrm{Fix}}}
\newcommand{\GammaFR}{\Gamma_{\mathrm{FR}}}

\newcommand{\A}{\mathbb{A}}
\newcommand{\C}{\mathbb{C}}
\newcommand{\m}{\mathfrak{m}}
\newcommand{\M}{\mathcal{M}}
\renewcommand{\O}{\mathcal{O}}
\renewcommand{\P}{\mathbb{P}}
\newcommand{\R}{\mathbb{R}}
\newcommand{\s}{\mathbf{s}}
\newcommand{\Z}{\mathbb{Z}}

\newcommand\tphi{{\widetilde{\varphi}}}
\newcommand\cO{{\mathcal O}}

\def\AA{{\mathbb A}}

\def\CC{{\mathbb C}}

\def\PP{{\mathbb P}}

\def\ZZ{{\mathbb Z}}

\renewcommand{\labelenumi}{(\Alph{enumi})}

\renewcommand{\tilde}{\widetilde}

\def\Char{{\rm{char}}}

\newcommand{\vv}{\vec{v}}

\begin{document}
\title{Configuration of the Crucial Set for A Quadratic Rational Map}

\author{John R. Doyle}
\address{John R. Doyle\\ 
Department of Mathematics\\
University of Rochester\\
Rochester, New York 16247\\
USA}
\email{john.doyle@rochester.edu}

\author{Kenneth Jacobs}
\address{Kenneth Jacobs\\ 
Department of Mathematics\\
University of Georgia\\
Athens, Georgia 30602\\
USA}
\email{kjacobs2@uga.edu}

\author{Robert Rumely}
\address{Robert Rumely\\ 
Department of Mathematics\\
University of Georgia\\
Athens, Georgia 30602\\
USA}
\email{rr@math.uga.edu}

\date{\today}
\thanks{Research for this paper was supported by NSF VIGRE grant DMS-0738586}

\begin{abstract}
Let $K$ be a complete, algebraically closed non-archimedean valued field, 
and let $\varphi(z) \in K(z)$ have degree two.  We describe the crucial
set of $\varphi$ in terms of the multipliers of $\varphi$ at the classical fixed points,
and use this to show that the crucial set determines a stratification 
of the moduli space $\M_2(K)$ related to the reduction type of $\varphi$.  
We apply this to settle a special case
of a conjecture of Hsia regarding the density of repelling periodic points in the 
non-archimedean Julia set.
\end{abstract}

\maketitle
\section{Introduction}

Let $K$ be an algebraically closed field, complete with respect to a non-Archimediean absolute value $|\cdot|_v$. 
Let $\mathcal{O}= \mathcal{O}_K$ denote its ring of integers and $\mathfrak{m}=\mathfrak{m}_K\subseteq \mathcal{O}$
 its maximal ideal. Let $k = \mathcal{O}/\mathfrak{m}$ denote the residue field. We assume that $|\cdot |_v$ 
and the logarithm $\log_v$ are normalized so that $\ord_{\mathfrak{m}} (x) = -\log_v |x|_v$. 
We will typically drop the dependence on $v$ and $\m$ in the notation and simply write $|\cdot |$ 
and $\ord$. Let $\pberk$ denote the Berkovich projective line over $K$;
it is a compact, uniquely path connected Hausdorff space which contains $\PP^1(K)$ as a dense subset.

Let $\varphi(z) \in K(z)$ be a rational map of degree $d \ge 2$. In \cite{Ru2}, Rumely found 
a canonical way to assign non-negative integer weights $w_{\varphi}(P)$ to points in $\pberk$, 
for which the sum of the weights is $d-1$. 
The set of points which receive weight is called the \emph{crucial set} of $\varphi$, and the probability measure
$\nu_{\varphi} := \frac{1}{d - 1} \sum_{P \in \, \pberk} w_{\varphi}(P) \delta_P(\cdot)$ 
is called the \emph{crucial measure} of $\varphi$.  When $\varphi$ has potential good reduction, the crucial set 
consists of the single point where $\varphi$ has good reduction.  Otherwise, the crucial set appears to classify the 
type of bad reduction that $\varphi$ has.  This paper provides quantitative support for that idea.

\smallskip

Recall that the points of $\pberk$ are 
said to be of types I -- IV.  Type I points are the ``classical'' points belonging to $\PP^1(K)$.
Type II points are the ``non-classical'' points $P$ where $\varphi$ has a meaningful reduction $\tphi_P \in k(z)$.
The crucial set is contained in the set of type II points; in particular, it lies in $\pberk \backslash \PP^1(K)$. 

There are four dynamical conditions, which we denote (W1) through (W4), 
under which a type II point $P \in \pberk$ carries weight (see Section~\ref{sect:notation} for formulas for the weights, 
and definitions of the terms below).  
Letting $\GammaF$ denote the tree in $\pberk$ spanned by the classical fixed points of $\varphi$, they are:

\renewcommand{\labelenumi}{(W\arabic{enumi})}
	\begin{enumerate}
		\item $P$ is a `Berkovich repelling fixed point' of $\varphi$;
		\item $P$ is a `Berkovich multiplicatively indifferent fixed point' of $\varphi$ \\
                         which is a branch point of $\GammaF$; 
		\item $P$ is a `Berkovich additively indifferent fixed point' of $\varphi$ belonging to $\GammaF$;
		\item $P$ is a branch point of $\GammaF$ which is moved by $\varphi$. 
	\end{enumerate}
	\renewcommand{\labelenumi}{(\Alph{enumi})}

\noindent{Here} we have added the word `Berkovich' for emphasis;  usually it will be omitted.
For a quadratic function, $d-1 = 1$, 
so there is a unique point $\xi$ which receives weight.  
It turns out that this is the Minimal Resultant Locus (see \cite{Ru1}), 
the point where $\varphi$ has ``best possible'' reduction.  
According as (W1) through (W4) holds at $\xi$, 
we will say that a quadratic map $\varphi$ has potential good reduction, potential multiplicative reduction,  
potential additive reduction, or potential constant reduction.

\smallskip  
In this paper, we first determine for which quadratic maps $\varphi$
the point $\xi$ satisfies (W1), (W2), (W3), or (W4); 
this is accomplished in Propositions \ref{prop:ramifiedfixedpt} and \ref{prop:main_prop}, 
where we show that the reduction type of $\varphi$ at $\xi$  
is determined by the multipliers at the classical fixed points. 

We apply this to study the image of $\varphi$ in $\M_2$, 
the moduli space of degree $2$ rational maps.  Using geometric invariant theory, 
Silverman \cite{silverman:1998} constructed $\M_d$ as a scheme over $\ZZ$ for all $d\geq 2$,
and for $d=2$ showed there is a canonical isomorphism $\s: \M_2 \to \A^2$. 
(Milnor had shown this earlier over $\C$; see \cite[Lemma 3.1]{milnor:1993}.) 
This leads to a natural compactification of $\M_2$ as $\overline{\M}_2 \cong \mathbb{P}^2$.  

For a quadratic map $\varphi(z) \in K(z)$, let $[\varphi] \in \M_2(K)$ denote the point corresponding to $\varphi$. 
We will base-change to $\cO$ and regard $\M_2$ and $\overline{\M}_2$ as schemes over $\cO$. 
The isomorphism $\s : \M_2 \to \A^2$ is given by $\s([\varphi]) = (\sigma_1(\varphi),\sigma_2(\varphi))$ where $\sigma_1, \sigma_2$
are the first and second symmetric functions in the multipliers at the fixed points of $\varphi$. 
We identify $\A^2(K)$ with  $\{\,[x:y:1] \} \subset \PP^2(K)$.
Given a point $P \in \P^2(K)$, we  write $\tilde{P} \in \P^2(k)$ for 
the specialization of $P$ modulo $\m$. 

For arbitrary $d \ge 2$, the connection between the crucial set and $\M_d$ was first noted in \cite{Ru2},
where it was shown that points in the barycenter of $\nu_\varphi$ correspond to conjugates of $\varphi$
having semi-stable reduction in the sense of geometric invariant theory. 
The main result of this paper is the following theorem, 
which says that for quadratic functions, 
the crucial set determines a stratification of $\M_2(K)$ compatible with 
specialization of $[\varphi]$ to $\overline{\M}_2(k)$:

\begin{thm}\label{thm:main}
Let $K$ be a complete, algebraically closed non-Archimedean field.  
Let $\varphi$ be a degree two rational map over $K$, and let $\xi$ denote the unique point in the crucial set of $\varphi$.
Then
	\begin{enumerate}
		\item $\tilde{\s([\varphi])} \in \A^2(k)$ if and only if $\xi$ satisfies $(W1)$ 
             $(\varphi$ has potential good reduction$)$.
		\item $\tilde{\s([\varphi])} = [\tilde{1} : \tilde{x} : \tilde{0}]$ 
                    for some $\tilde{x} \in k$ with $\tilde{x} \ne \tilde{2}$ 
                   if and only if $\xi$ satisfies $(W2)$ $(\varphi$ has potential multiplicative reduction$)$. 
                    In this case, $\tilde{x} = \tilde{\lambda} + \tilde{\lambda}^{-1}$, 
                    where $\lambda$ is the multiplier of $\varphi$ at a classical indifferent fixed point.
		\item $\tilde{\s([\varphi])} = [\tilde{1} : \tilde{2} :\tilde{0}]$ if and only if $\xi$ satisfies $(W3)$
                    $(\varphi$ has potential additive reduction$)$.
		\item $\tilde{\s([\varphi])} = [\tilde{0} : \tilde{1} : \tilde{0}]$ if and only if $\xi$ satisfies $(W4)$
                    $(\varphi$ has potential constant reduction$)$.
	\end{enumerate}
\end{thm}

The fact that $\tilde{\s([\varphi])} \in \A^2(k)$ if and only if $\varphi$ has potential good reduction 
had previously been shown by D. Yap in her thesis \cite[Thm. 3.0.3]{yap:2012}. 
Our theorem may be considered a strengthening of Yap's result.  One also notes the parallel between 
Theorem \ref{thm:main} and Milnor's description \cite{milnor:1993} of degenerations of quadratic 
maps over $\CC$, as they approach the boundary of moduli space. 

In \textsection \ref{sec:hsia}, we observe that Theorem~\ref{thm:main} implies the following result, 
which resolves a special case of a conjecture of L.-C. Hsia (\cite[Conj. 4.3]{hsia:2000}).

\begin{prop}\label{prop:main}
Let $\varphi$ be a quadratic rational map defined over $K$, let $\mathcal{J}_\varphi(K) \subseteq \P^1(K)$ 
be the $($classical$)$ Julia set of $\varphi$, and let $\overline{\mathcal{R}_\varphi(K)}$ be the closure 
in $\PP^1(K)$ of the set of type {\rm{I}} repelling periodic points for $\varphi$. 
Then $\mathcal{J}_\varphi(K) = \overline{\mathcal{R}_\varphi(K)}$.
\end{prop}

\subsection{Outline of the Paper}
In Section~\ref{sect:notation}, we introduce notation and concepts used in the rest of the article. 
In particular, we give a more detailed explanation of the weights $w_\varphi$ and the conditions (W1)---(W4) 
under which a point can have weight. In Section~\ref{sect:structures} we relate the reduction type of the 
unique weighted point $\xi$ to the multipliers at the classical fixed points. 
For this, we rely on two normal forms for quadratic rational maps given in \cite{silverman:2012}. 
In Section~\ref{sect:modulispace} we apply our analysis to prove Theorem \ref{thm:main}, 
 and give the application to Hsia's conjecture.

\subsection{Acknowledgements}
The research for this article was begun during an NSF-sponsored %\footnote{The funding was provided by a VIGRE Grant DMS-0738586.}
VIGRE research seminar on dynamics on the Berkovich line at the University of Georgia, led by the third author. We would like to thank the other members of our research group for many helpful discussions, and would especially like to thank Jacob Hicks, Allan Lacy, Marko Milosevic, and Lori Watson for their insights and contributions to this work.

\section{Notation and Conventions}\label{sect:notation}
In this section we introduce terminology and notation used throughout the paper.

\subsection{Berkovich Space}
Formally, the Berkovich \emph{affine} line $\aberk$ over $K$ is the collection of equivalence classes of multiplicative seminorms on $K[T]$ 
which extend the norm $|\cdot|_v$ on $K$. Berkovich showed (\cite{Ber}, p.18) that each such seminorm $[\cdot]_x$ corresponds to a decreasing 
sequence of discs $\{D(a_i, r_i)\}$ in $K$ (more precisely, to a cofinal equivalence class of such sequences) via the correspondence 
$$[f]_x := \lim_{i\to\infty} [f]_{D(a_i, r_i)}\ .$$ 
Here, $[f]_{D(a,r)} = \sup_{z\in D(a,r)} |f(z)|$ is the sup-norm on the disc $D(a,r)$. With this,
we obtain a classification of the points of $\aberk$ into four types:

\begin{itemize}
\item \emph{Points of type I} correspond to nested, decreasing sequences of discs $\{D(a_i, r_i)\}$ whose intersection is a single point in $K$; formally, these are the seminorms $[f]_{a} = |f(a)|$ for $a\in K$.
\item \emph{Points of type II} correspond to nested, decreasing sequences of discs whose intersection is a disc $D(a,r)\subseteq K$ with $r\in |K^\times|$. In this case, we have $[f]_x = \sup_{z\in D(a,r)} |f(z)|$.  Note that for polynomials $f\in K[T]$, the supremum is achieved at some point in $D(a,r)$. 
\item \emph{Points of type III} correspond to nested, decreasing sequences of discs whose intersection is a disc $D(a,r) \subseteq K$ with $r\not\in |K^\times|$; as in the case of type II points, the corresponding seminorm is the sup-norm on $D(a,r)$. In this case, the $\sup$ is not achieved unless $f$ is constant.
\item \emph{Points of type IV} correspond to nested, decreasing sequences of discs $\{D(a_i, r_i)\}$ whose intersection is empty, but for which $\lim_{i\to\infty} r_i >0$. Such points can occur only if the field $K$ is not spherically complete.
\end{itemize}

One often writes type I, II, and III points in terms of their corresponding discs $D(a,r)$ using the shorthand
$\zeta_{D(a,r)}$ or simply $\zeta_{a,r}$. The point corresponding to the unit disc is called the Gauss point 
(it corresponds to the Gauss norm on polynomials), and is written $\zetaG = \zeta_{D(0,1)}$. 

The construction of $\pberk$ from $\aberk$ is similar to the construction of $\mathbb{P}^1$ from $\mathbb{A}^1$, 
gluing two copies of $\aberk$ together by means of an involution of $\aberk \setminus \{0\}$; 
see Section 2.2 of \cite{BR} for details.  We write $\hberk$ for $\pberk \backslash \PP^1(K)$, 
the `non-classical' part of $\pberk$.

The Berkovich Line is typically endowed with the Berkovich-Gel'fand topology, which is the weakest topology 
for which the map $x\mapsto [f]_x$ is continuous for every $f\in K[T]$. In this topology, $\pberk$ is a compact Hausdorff space and is uniquely path connected. The points of type I, II and III are dense in $\pberk$ for this topology. In general the Berkovich-Gel'fand topology is not metrizable. 

A rational map $\varphi\in K(z)$ induces a continuous action on $\mathbb{P}^1(K)$ by means of a lift $\Phi = [F,G]$, 
where $F,G\in K[X,Y]$ are homogeneous polynomials of degree $d = \deg(\varphi)$ 
such that $\varphi(z) = \frac{F(z,1)}{G(z,1)}$. 
This action extends continuously to all of $\pberk$, and preserves types of points.
 One can show that $\PGL_2(K)$ acts transitively on type II points, and that $\PGL_2(\mathcal{O})$ is the stabilizer of the Gauss point.

\subsubsection{Tree Structure}
The Berkovich projective line can also be viewed as a tree. The collection of points $\{\zeta_{a,r}\}_{r\in [t,s]}\subseteq \pberk$ 
is naturally homeomorphic to the real segment $[t,s]$. The type II points are dense along such a segment, and at any type II point 
there are infinitely many branches away from $\{\zeta_{a,r}\}_{r\in [t,s]}$;  
indeed the branches are in $1-1$ correspondence with the elements of $\PP^1(k)$ for the residue field $k$.  
To make this clearer, consider the type II point $\zetaG = \zeta_{D(0,1)}$. 
The branches off $\zetaG$ come from equivalence classes of paths $[\zetaG,x]$ sharing a common initial segment; 
these classes correspond to subdiscs  $D(b,1)^- = \{x \in K : |x-b| < 1\}$ where $|b| \le 1$, and to the set $\PP^1(K) \backslash D(0,1)$.
Identifying $D(0,1)$ with the valuation ring $\cO_K$, the subdiscs $D(b,1)^- $ are just the cosets 
$b + \mathfrak{m}$ in $k = \cO_K/\mathfrak{m}$, and $\PP^1(K) \backslash D(0,1)$ corresponds to $\infty \in \PP^1(k)$.  

A more geometric way to think of the branches is in terms of tangent directions. Formally, a tangent 
direction $\vv$ at $P$ is an equivalence class of paths emanating from $P$. The collection of tangent directions at $P$ will be 
denoted by $T_P$. For points of type II, $T_P$ is in $1-1$ correspondence with $\mathbb{P}^1(k)$ as noted above. 
For points of type III, $T_P$ consists of two directions, 
while for points of type I and IV, $T_P$ consists of the unique 
direction pointing into $\pberk$.
If $P, Q$ are points of $\pberk$ with $\varphi(P) = Q$, there is 
a canonical induced surjective map $\varphi_* : T_P \rightarrow T_Q$.

\subsubsection{Reduction of Rational Maps}  
The action of $\varphi$ on the tangent space $T_P$ is closely related to the notion of the reduction of $\varphi$,  
which we describe here. If $[F,G]$ is a lift of $\varphi$ that has been scaled so that the coefficients all lie 
in $\mathcal{O}$, and so that at least one is a unit, we call $[F,G]$ a {\em normalized lift},
or {\em normalized representation}, of $\varphi$. 
Such a representation is unique up to scaling by a unit in $\cO$. We can reduce each coefficient of a normalized 
lift $[F, G]$ modulo $\mathfrak{m}$.  After removing common factors, we obtain a well-defined map 
$[\tilde{F}: \tilde{G}]$ on $\mathbb{P}^1(k)$. This map, called the reduction of $\varphi$ at $\zeta_G$, 
is denoted $\tilde{\varphi}$.  

If $P$ is an arbitrary type II point, there is a $\gamma \in \PGL_2(K)$ for which $\gamma(\zeta_G) = P$;  we define the 
reduction of $\varphi$ at $P$ to be the reduction of the conjugate $\varphi^\gamma = \gamma^{-1} \circ \varphi \circ \gamma$: 
\begin{equation*}
\tilde{\varphi}_P(z) \ := \ \tilde{\varphi^{\gamma}}(z) \ .
\end{equation*} 
The reduction $\tilde{\varphi}_P$ is unique up to conjugation by an element of $\PGL_2(k)$;  in particular, the degree $\deg(\tilde{\varphi}_P)$
is well-defined.  

It was shown by Rivera-Letelier that a type II point $P\in \hberk$ is fixed by $\varphi$ 
if and only if the reduction $\tilde{\varphi}_P$ is non-constant (see \cite{BR}, Lemma 2.17);  
equivalently, $\varphi(P) \ne P$ if and only if $\tilde{\varphi}_P$ is constant.
Rivera-Letelier calls a type II point $P$ a {\em repelling fixed point} if $\deg(\tilde{\varphi}_P) \ge 2$, 
and he calls $P$ an {\em indifferent fixed point} if  $\deg(\tilde{\varphi}_P) = 1$.
Rumely \cite[Def. 2]{Ru2} gave a refined classification of indifferent fixed points in $\hberk$:

\begin{defn}
If $P$ is a type {\rm II} indifferent fixed point of $\varphi$, 
then after a change of co\"ordinates on $\mathbb{P}^1(k)$, exactly one of the following holds:

\begin{itemize}
\item $\tilde{\varphi}_P(z) = \tilde{c}z$ for some $\tilde{c}\in k^{\times}$, $\tilde{c} \ne \tilde{1}$, 
in which case we say $P$ is a {\rm (Berkovich) multiplicatively indifferent} fixed point for $\varphi$.
\item $\tilde{\varphi}_P(z) = z+\tilde{a}$ for some $\tilde{a}\in k^{\times}$, 
in which case we say $P$ is an {\rm (Berkovich) additively indifferent} fixed point for $\varphi$.
\item $\tilde{\varphi}_P(z) = z$,  in which case we say $P$ is an {\rm id-indifferent} fixed point for $\varphi$.
\end{itemize}
\end{defn}

One should think of each of the above reduction types as describing the behavior of the map $\varphi_*$ acting on $T_P$. 
More precisely, after conjugating $\varphi$ by a suitable $\gamma \in \PGL_2(K)$ we can assume that $P = \zetaG$ is fixed.  
Then $\tilde{\varphi}$ is a well-defined non-constant map, and by making use of the identification 
$T_P \cong \mathbb{P}^1(k)$, if $\vv_a\in T_P$ corresponds to the point $a\in \mathbb{P}^1(k)$, 
then $\varphi_*(\vv_{a}) = \vv_{\tilde{\varphi}(a)}$. 

It is common to say that a map $\varphi(z) \in K(z)$ has {\em good reduction} if $\deg(\tilde{\varphi}) = \deg(\varphi)$,
and {\em potential good reduction} if $\deg(\tilde{\varphi}_P) = \deg(\varphi)$ for some $P \ne \zeta_G$.
However, in this paper we will not distinguish good reduction from potential good reduction.
By the discussion above, a quadratic map $\varphi$ has potential good reduction iff  it has a type II repelling fixed point. 

\subsection{The Crucial Set}\label{sect:classification}

The crucial set was constructed in \cite{Ru2}, and arose from the study of a certain function 
$\ord\Res_\varphi : \pberk \to \R \cup \{\infty\}$. This function had been introduced in \cite{Ru1}, in order 
to address the question of finding conjugates $\varphi^\gamma$ that had minimal resultant.
One obtains the crucial measure and crucial set by taking the graph-theoretic Laplacian of $\ord\Res_\varphi(\cdot)$, 
restricted to a canonical tree $\GammaFR \subset \pberk$.

In this section, we briefly sketch this construction.

\subsubsection{The Function $\ord\Res_{\varphi}(x)$}
Let $\varphi\in K(z)$ have degree $d \geq 2$, and let $[F,G]$ be a lift of $\varphi$.  Writing 
\begin{align*}
F(X,Y) &= a_0 X^d +a_1X^{d-1}Y+ ... + a_d Y^d\\ 
G(X,Y) & = b_0 X^d + b_1 X^{d-1}Y+... +b_dY^d
\end{align*} 
put $\ord(F) = \min_{0 \leq i \leq d} (\ord(a_i))$ and $\ord(G) = \min_{0 \leq i \leq d} (\ord(b_i))$. 
If $\max(|a_i|, |b_i|) = 1$, (that is, $\min(\ord(F),\ord(G)) = 0$), 
the lift is normalized. For $\gamma = 
	\begin{pmatrix}
		A & B\\
		C & D
	\end{pmatrix} \in \PGL_2(K)$, we define
\begin{align*}
	F^{\gamma}(X,Y) &:= DF(AX+BY,CX+DY) - BG(AX+BY,CX+DY),\\
	G^{\gamma}(X,Y) &:=-CF(AX+BY,CX+DY) + AG(AX+BY,CX+DY),
\end{align*}
so that $\Phi^\gamma := [F^\gamma,G^\gamma]$ is a lift of $\varphi^{\gamma}$.

The resultant of the lift $\Phi=[F,G]$ is the determinant of the Sylvester matrix:

\begin{align*}
\Res(\Phi):= \Res(F,G) = & \det\left(
\begin{array}{cccccccc} 
a_0 & a_1 & \dots & a_{d-1} &a_d &0 &  \dots & 0 \\ 
0 & a_0 & a_1 & \dots & a_{d-1} & a_d & \dots & 0 \\
\vdots&& & \ddots & \vdots & \vdots & & \vdots \\
0 & 0 & 0 & a_0 & a_1 & \dots & a_{d-1} & a_d \\
b_0 & b_1 & \dots & b_{d-1} &b_d &0 &  \dots & 0 \\ 
0 & b_0 & b_1 & \dots & b_{d-1} & b_d & \dots & 0 \\
\vdots & & & \ddots & \vdots & \vdots & & \vdots \\
0 & 0 & 0 & b_0 & b_1 & \dots & b_{d-1} & b_d \\
\end{array}\right)\ .
\end{align*}
Let $\zeta\in \pberk$ be a type II point, and choose $\gamma \in \PGL_2(K)$ so that $\zeta = \gamma(\zetaG)$. 
Fix a \emph{normalized} lift $\Phi^\gamma$ of $\varphi^\gamma$. We then define 
\begin{equation*}
\ord\Res_{\varphi}(\zeta) \ := \ \ord(\Res(\Phi^\gamma)) \ .
\end{equation*} 
Using standard formulas for the resultant from \cite{silverman:2007}, 
one sees that $\ord\Res_\varphi(\zeta)$ is well-defined, and that  
\begin{equation*}
\ord\Res_{\varphi}(\zeta) 
\ = \ \ord\Res_\varphi(\zeta_G) + (d^2+d) \ord\big(\det(\gamma)\big) 
                       - 2d \min\big(\ord(F^\gamma), \ord(G^\gamma)\big)\ ,
\end{equation*} 
(The `min' term 
assures we are using a normalized lift $\Phi^\gamma$.) 
It is shown in \cite{Ru1} that the function $\ord\Res_{\varphi}$ 
on type II points extends to a continuous function 
$\ord \Res_\varphi : \pberk \rightarrow [0,\infty]$, which, for any $a \in \PP^1(K)$,   
attains its minimum on the tree $\Gamma_{\mathrm{Fix},\varphi^{-1}(a)}$ 
spanned by the classical fixed points and the pre-images 
under $\varphi$ of $a$.  
It is also shown in \cite{Ru2} that the tree 
$\GammaFR$ spanned by the classical fixed points and the repelling fixed points in $\hberk$
is the intersection of all the trees $\Gamma_{\mathrm{Fix},\varphi^{-1}(a)}$: 
\begin{equation} \label{TreeIntersectionThm}
		\GammaFR \ = \  \bigcap_{a \in \P^1(K)} \Gamma_{\mathrm{Fix},\varphi^{-1}(a)} \ .
\end{equation} 
This is useful in determining $\GammaFR$.

\subsubsection{The Crucial Measures}

The crucial measure is obtained by taking the graph-theoretic Laplacian of $\ord\Res_\varphi(\cdot)$ 
on (a suitable truncation\footnote{In order to apply the theory of graph Laplacians, one must first `prune' the tree $\GammaFR$ to remove its type I endpoints --- see \cite[p. 25]{Ru2}. We omit the details here, as they won't be necessary in this article.} of) the tree $\GammaFR$. More precisely, if $\mu_{Br}$ is the `branching measure' which gives each $P \in \GammaFR$ the weight $1-\frac{1}{2} v(P)$, where $v(P)$ is the valence of $P$ in $\GammaFR$, then

\begin{defn}[{Rumely, \cite[Cor. 6.5]{Ru2}}] The crucial measure associated to $\varphi$ is the  
measure $\nu_\varphi$ on $\GammaFR$ defined by 
\begin{equation*}
\Delta_{\GammaFR}(\ord\Res_{\varphi}(\cdot)) = 2(d^2-d)(\mu_{Br} - \nu_{\varphi}) \ .
\end{equation*} 
It is a probability measure with finite support, and its support is contained in $\hberk$.
\end{defn} 

\noindent{The crucial measure} is canonically attached to $\varphi$, because the function $\ord \Res_\varphi$ 
and the tree $\GammaFR$ are canonical.  It is a conjugation equivariant of $\varphi$ in $\hberk$, 
just as the sets of classical fixed points and critical points are conjugation equivariants in $\PP^1(K)$.  

Rumely gave an explicit expression for $\nu_\varphi$ as a sum of weighted point masses:
\begin{equation*}
\nu_\varphi \ = \ \frac{1}{d-1} \sum_{P \in \, \pberk} w_\varphi(P) \, \delta_P(\cdot) \ , 
\end{equation*}
where the weights $w_\varphi(P)$ are as follows (\cite[Def. 8]{Ru2}):

\begin{defn}\label{def:weight}
For a point $P \in \hberk$, if $P$ fixed by $\varphi$, let $N_{\textrm{shearing}, \varphi}(P)$ 
be the number of directions $\vv\in T_P$ that contain type {\rm I} fixed points 
but are moved by $\varphi_*$. Let $v(P)$ denote the valence of $P$ in $\GammaFR$  
$($set $v(P) = 0$ if $P \not \in \GammaFR)$.
Then the weight $w_{\varphi}(P)$ of a point $P\in \pberk$ is as follows: 
\begin{enumerate}
\item If $P$ is a type {\rm II} fixed point of $\varphi$, 
then $w_{\varphi}(P)= \deg(\widetilde{\varphi}_P) - 1 + N_{\textrm{shearing}, \varphi}(P)$.
\item If $P$ is a branch point of $\Gamma_{\rm{Fix}}$ which is moved by $\varphi$ 
$($necessarily of type {\rm II}$)$, then $w_{\varphi}(P) = v(P)-2$.
\item Otherwise, $w_{\varphi}(P) = 0$.
\end{enumerate}
\end{defn}

\noindent{The above formulas for the weights} give rise to conditions (W1)--(W4)  
in the Introduction. 

The fact that $\nu_\varphi$ is a probability measure is equivalent to the following formula:

\begin{thm}[Rumely \cite{Ru2}, Theorem 6.2]
Let $\varphi\in K(z)$ have degree $d\geq 2$. Then 
\begin{equation} \label{WeightFormula}
\sum_{P\in \pberk} w_{\varphi}(P) \ = \ d-1 \ .
\end{equation}
\end{thm}

\noindent{We emphasize} that for a quadratic rational map, formula (\ref{WeightFormula}) implies there 
is a unique point $\xi\in \pberk$ with $w_\varphi(\xi) >0$; it is the behaviour of $\varphi$ 
at this point that we plan to study.

\subsection{The Moduli Space of Quadratic Rational Maps}

Using geometric invariant theory, Silverman \cite[Thm. 1.1]{silverman:1998} constructed the moduli space $\M_d$ 
for rational maps of degree $d \ge 2$ as a scheme over $\Z$.  He also showed \cite[Thm. 5.1]{silverman:1998} 
there is a natural isomorphism $\s: \M_2 \to \A^2$ as schemes over $\Z$.  More precisely, he showed that the first and second elementary 
symmetric functions $\sigma_1$, $\sigma_2$ of the multipliers at the fixed points give co\"ordinates on $\M_2$. 

This means that if $\varphi(z) \in K(z)$ is a quadratic map with fixed points $\alpha_1, \alpha_2, \alpha_3$ 
(listed with multiplicity)
and corresponding multipliers $\lambda_1, \lambda_2, \lambda_3$, and if we put 
\begin{equation*}
\sigma_1(\varphi) = \lambda_1 + \lambda_2 + \lambda_3 \ , \qquad 
\sigma_2(\varphi) = \lambda_1 \lambda_2 + \lambda_1 \lambda_3 + \lambda_2 \lambda_3 \ ,
\end{equation*}  
then the point $[\varphi]$ in $\M_2(K) \cong \AA^2(K)$  corresponding to $\varphi$ 
is $(\sigma_1(\varphi), \sigma_2(\varphi))$.
 
\section{The Crucial Sets of Quadratic maps}\label{sect:structures}

The behavior of a rational map $\varphi$ near a classical fixed point $\alpha \in \PP^1(K)$ is governed by the multiplier 
at $\alpha$. In this section, we explicitly describe the crucial set for quadratic rational maps in terms of 
the multipliers at the classical fixed points. 

If $\alpha \in K$ is a fixed point for the rational map $\varphi$, 
the derivative $\varphi^{\prime}(\alpha)$ is called the multiplier of $\varphi$ at $\alpha$.  
It is well-known that the multiplier 
is independent of the choice of coordinates, which means the multiplier at $\varphi$ at $\infty \in \PP^1(K)$ 
can be defined by changing coordinates.
  
Letting $\lambda$ be the multiplier at $\alpha$, one says that $\alpha$ is

\begin{center}
	\begin{tabular}{r l}
		\emph{attracting}, & if $|\lambda| < 1$;\\
		\emph{indifferent}, & if $|\lambda| = 1$; and\\
		\emph{repelling}, & if $|\lambda| > 1$.
	\end{tabular}
\end{center}
Throughout this section, we will let $\alpha_1,\alpha_2,\alpha_3$ be the (not necessarily distinct) 
fixed points for $\varphi$, and we will let $\lambda_1,\lambda_2,\lambda_3$ be the corresponding multipliers.

\subsection{Maps with a Multiple Fixed Point} \label{MultipleFixedPtSubsection}
We begin our classification by considering quadratic rational maps $\varphi$ with a multiple fixed point. 
In this case, we may assume without loss of generality that $\alpha_1 = \alpha_2$, 
which means that necessarily $\lambda_1 = \lambda_2 = 1$. 
By \cite[Lem. 2.46]{silverman:2012}, $\varphi$ is conjugate to the rational map
\begin{equation*}
	z \ \mapsto \ z+\sqrt{1-\lambda_3}+\frac{1}{z} \ ,
\end{equation*}
with fixed points $\alpha_1 = \alpha_2 = \infty$ and $\alpha_3 = -\dfrac{1}{\sqrt{1-\lambda_3}}$,  
and multipliers $\lambda_1 = \lambda_2 = 1$ and $\lambda_3$.

\begin{prop}\label{prop:ramifiedfixedpt}
Suppose
	\[ \varphi(z) = z+\sqrt{1-\lambda_3}+\frac{1}{z} \ , \]
and let $\xi$ be the unique point in $\pberk$ with $w_{\varphi}(\xi) = 1$.

\begin{enumerate}
\item If $|\lambda_3| \leq 1$, then $\xi=\zetaG$ satisfies $(W1)$ $(\varphi$ has potential good reduction$)$.

\item If $|\lambda_3| > 1$, then $\xi=\zeta_{D\left(0, \sqrt{|\lambda_3|}\right)}$ satisfies $(W3)$ 
$(\varphi$ has potential additive reduction$)$.
\end{enumerate}

\end{prop}

\begin{proof}
We immediately see that if $|\lambda_3| \leq 1$ then $|1-\lambda_3| \leq 1 $ 
so $\varphi(z)$ has good reduction at $\zetaG$, proving (A). We therefore suppose that $|\lambda_3| > 1$. Using (\ref{TreeIntersectionThm})
we find that
	\[ \GammaFR \ = \ \GammaF \ = \ \left[ -\frac{1}{\sqrt{1-\lambda_3}}, \infty\right]\ .\]
In particular, since $|\lambda_3| > 1$, we have $\zeta_{D\left(0, \sqrt{|\lambda_3|}\right)} \in \GammaFR$.

Conjugating $\varphi$ by $\gamma(z) = \sqrt{1 - \lambda_3}\cdot z$ we find that 
	\[
		\varphi^{\gamma}(z) \ = \ \dfrac{z^2 + z + \dfrac{1}{1-\lambda_3}}{z} \ .
	\]
Since $|1-\lambda_3| = |\lambda_3| > 1$, reducing modulo $\m$ yields
	\[
		\widetilde{\varphi^\gamma}(z) \ = \ \frac{z^2 + z}{z} \ = \ z + \tilde{1} \ ,
	\]
which shows $\varphi$ satisfies (W3) at the point 
$\xi = \gamma(\zetaG) = \zeta_{D\left(0, |\sqrt{1 - \lambda_3}|\right)} = \zeta_{D\left(0,\sqrt{|\lambda_3|}\right)}$.

\end{proof}

\subsection{Maps With Distinct Fixed Points}
We now turn to quadratic rational maps with three distinct classical fixed points. 
In this case, the multiplier of the third fixed point is determined by the multipliers of the other two:

\begin{lemma}\label{lem:lambda_3}
Let $\varphi$ be a degree two rational map with three distinct classical fixed points. Let $\lambda_1$ and $\lambda_2$ be the multipliers of two of the fixed points. Then the third  has multiplier
	\begin{equation}\label{eq:lambda_3}
	\lambda_3 \ = \ \frac{\lambda_1 + \lambda_2 - 2}{\lambda_1\lambda_2 - 1}.
	\end{equation}
\end{lemma}

\begin{proof}
Since $\varphi$ has three distinct fixed points, none of the multipliers can be equal to one. 
Therefore, we have the well-known formula (see \cite[Theorem 1.14]{silverman:2007})
\begin{equation} \label{FixedPtResidueFormula} 
 \frac{1}{1-\lambda_1} + \frac{1}{1-\lambda_2} + \frac{1}{1-\lambda_3}\ = \ 1 \ . 
\end{equation} 
Solving for $\lambda_3$ yields the desired result.
\end{proof}

\begin{lemma}\label{lem:repelling}
Let $\varphi$ be a degree two rational map over $K$ with three distinct classical fixed points. 
Then these cannot all be repelling.
Moreover, 
	\begin{enumerate}
		\item if $\varphi$ has two classical repelling fixed points, then the third is attracting$;$
		\item if $\varphi$ has only one classical repelling fixed point, then the other two are indifferent$;$
        \item if $\varphi$ has no classical repelling fixed points, then either some pair of multipliers satisfies 
                             $\tilde{\lambda}_i \tilde{\lambda}_j \ne \tilde{1}$, 
                        or else $\tilde{\lambda}_1 = \tilde{\lambda}_2 = \tilde{\lambda}_3 = \tilde{1}$.
	\end{enumerate}
\end{lemma}

\begin{proof}
Suppose that $\varphi$ has two repelling fixed points, say $\alpha_1$ and $\alpha_2$,  
with multipliers $\lambda_1$ and $\lambda_2$. 
By \eqref{eq:lambda_3}, the multiplier $\lambda_3$ of $\alpha_3$ satisfies 
	\[ |\lambda_3| \ = \ \frac{|\lambda_1 + \lambda_2 - 2|}{|\lambda_1\lambda_2 - 1|} \ \le \ 
                             \frac{\max\{|\lambda_1|,|\lambda_2|\}}{|\lambda_1\lambda_2|} \ < \ 1\ , \]
so $\alpha_3$ is attracting. This shows that $\varphi$ cannot have three repelling fixed points, 
and also proves (A).

To show (B), suppose the fixed points of $\varphi$ are labeled so that 
	\[ |\lambda_1| \ > \ 1 \ \ge \ |\lambda_2| \ \ge \ |\lambda_3|\ . \]
Again using \eqref{eq:lambda_3}, we have
	\[ |\lambda_3| \ = \ \frac{|\lambda_1 + \lambda_2 - 2|}{|\lambda_1\lambda_2 - 1|} \ \ge \ 
     \frac{|\lambda_1|}{\max\{|\lambda_1\lambda_2|, 1\}} \ \ge \ \frac{|\lambda_1|}{\max\{|\lambda_1|, 1\}} \ = \ 1 \ . \]
Since we assumed that $|\lambda_3| \le 1$, equality holds throughout, and therefore $|\lambda_2| = |\lambda_3| = 1$. 

To show (C), suppose that $|\lambda_1|, |\lambda_2|, |\lambda_3| \le 1$.  If some pair of multipliers
satisfies $\tilde{\lambda}_i \tilde{\lambda}_j \ne \tilde{1}$, we are done.  Otherwise 
$\tilde{\lambda}_1 \tilde{\lambda}_2 = \tilde{\lambda}_1 \tilde{\lambda}_3 
= \tilde{\lambda}_2 \tilde{\lambda}_3 = \tilde{1}$. 
Considering these equalities in pairs, we conclude there is a $\tilde{c} \in k$ such that 
$\tilde{\lambda}_1 = \tilde{\lambda}_2 = \tilde{\lambda}_3 = \tilde{c}$.
In particular, we have $\tilde{c}^2 = \tilde{1}$, so $\tilde{c} \in \{\pm \tilde{1}\}$. If $\tilde{c} = \tilde{1}$ (which is necessarily true if $\Char(k) = 2$), then we are done, so assume that $\Char(k) \ne 2$ and $\tilde{c} = -\tilde{1}$. In this case, reducing \eqref{FixedPtResidueFormula} modulo $\mathfrak{m}$ yields the equation $\tilde{\left(3/2\right)} = \tilde{1}$, which implies that $\tilde{2} = \tilde{3}$, a contradiction. Hence $\tilde{\lambda}_1 = \tilde{\lambda}_2 = \tilde{\lambda}_3 = \tilde{1}$.

\end{proof}

Our next result is parallel to Proposition~\ref{prop:ramifiedfixedpt} and describes the structure of the crucial set for a quadratic rational map with three distinct fixed points. 
For such maps, we know from \cite[Lemma 2.46]{silverman:2012}
that $\varphi$ is conjugate to a map of the form
\begin{equation*}
	z \ \mapsto \ \frac{z^2 + \lambda_1z}{\lambda_2z + 1} \ ,
\end{equation*}
where $\lambda_1$ and $\lambda_2$ are two of the fixed point multipliers for $\varphi$. 
We will henceforth assume $\varphi$ is given in this form. The fixed points of $\varphi$ are then 
$\alpha_1 = 0$, $\alpha_2 = \infty$, and $\alpha_3 = (\lambda_1 - 1)/(\lambda_2 - 1)$,
 with multipliers $\lambda_1$, $\lambda_2$, and $\lambda_3$, respectively. 
This means that $\Gamma_{\textrm{Fix}}$ has a single branch point at $\zeta_{D(0,|\alpha_3|)}$.

Furthermore, if 
$\varphi$ has no repelling classical fixed points, 
i.e., if $|\lambda_1|, |\lambda_2|, |\lambda_3| \le 1$, then by Lemma \ref{lem:repelling} either 
we can conjugate $\varphi$ so that $\tilde{\lambda_1 \lambda_2} \ne 1$, or 
$\tilde{\lambda}_1 = \tilde{\lambda}_2 = \tilde{\lambda}_3 = \tilde{1}$.  
On the other hand, if $\varphi$ has a repelling classical fixed point, then by Lemma \ref{lem:repelling} it 
also has a non-repelling classical fixed point;  hence by conjugating $\varphi$ if necessary,
we can assume that $|\lambda_1| > 1 \ge |\lambda_2|$. 

\begin{prop}\label{prop:main_prop}
Let
	\[ \varphi(z) \ = \ \frac{z^2 + \lambda_1z}{\lambda_2z + 1}, \]
and let $\xi$ be the unique point in $\pberk$ with $w_\varphi(\xi) = 1$.
\begin{enumerate}
\item If $\varphi$ has no repelling classical fixed points, 
           then $\xi$ satisfies {\rm(W1)} $(\varphi$ has potential good reduction$).$ 
            Replacing $\varphi$ by a conjugate if necessary, we can assume that either 
            $\tilde{\lambda_1 \lambda_2} \ne 1$, or that 
            $\tilde{\lambda}_1 = \tilde{\lambda}_2 = \tilde{\lambda}_3 = \tilde{1}$.  In this setting,  
	\begin{enumerate}
		\item if $\tilde{\lambda_1\lambda_2} \ne \tilde{1}$, then $\xi = \zetaG;$
		\item if  $\tilde{\lambda}_1 = \tilde{\lambda}_2 = \tilde{\lambda}_3 = \tilde{1}$, 
            then $\xi = \zeta_{D\left(-1,\sqrt{|\lambda_1\lambda_2 - 1|}\right)}$.
	\end{enumerate}

	\item Suppose $\varphi$ has at least one repelling classical fixed point, 
hence also a non-repelling fixed point by Lemma~$\ref{lem:repelling}$. 
Replacing $\varphi$ by a conjugate if necessary, we can assume that $|\lambda_1| > 1 \ge |\lambda_2|$. 
In this situation $\xi = \zeta_{D(0,|\lambda_1|)}$. Moreover, 
	\begin{enumerate}
		\item if $\tilde{\lambda_2} \not \in \{\tilde{0},\tilde{1}\}$, then $\xi$ satisfies {\rm(W2)} 
             $(\varphi$ has potential multiplicative reduction$);$
		\item if $\tilde{\lambda_2} = \tilde{1}$, then $\xi$ satisfies {\rm(W3)}
             $(\varphi$ has potential additive reduction$);$
		\item if $\tilde{\lambda_2} = \tilde{0}$, then $\xi$ satisfies {\rm(W4)}
		     $(\varphi$ has potential constant reduction$).$ 
	\end{enumerate}
\end{enumerate}
\end{prop}

\begin{rem}
Note that in case (A)(ii), the point $\xi$ is different from $\zetaG$, 
since $|\lambda_1\lambda_2 - 1| < 1$ and therefore
	\[ D\left(-1,\sqrt{|\lambda_1\lambda_2 - 1|}\right) \ \subsetneq \ D(-1,1) \ = \ D(0,1) \ . \]
This is the only situation where $\GammaFR $ may be strictly larger than $\Gamma_{\mathrm{Fix}}$.  
\end{rem}

\begin{proof}
First, assume that $\varphi$ has no classical repelling fixed points, 
which means that each of the multipliers lies in $\O$. 
In particular, this implies that the expression for $\varphi$ given in the proposition is already normalized. 
Since $\Res(\Phi) = 1 - \lambda_1\lambda_2$, we see that if $\tilde{\lambda_1\lambda_2} \ne \tilde{1}$, 
then $|\Res(\Phi)| = |1 - \lambda_1\lambda_2| = 1$, so $\varphi$ has good reduction at $\zeta_G$, 
proving (A)(i).

Now suppose that $\tilde{\lambda}_1 = \tilde{\lambda}_2 = \tilde{\lambda}_3 = \tilde{1}$. 
This means that $|\lambda_1\lambda_2 - 1| < 1$.
Set $\rho := \sqrt{\lambda_1\lambda_2 - 1}$, and let $r  := |\rho| < 1$. Set $\gamma(z) := \rho z - 1$, so that $\gamma(\zetaG) = \zeta_{D(-1,r)}$. To prove (A)(ii), it suffices to show that $\varphi^{\gamma}$ has good reduction.

The map $\varphi^{\gamma}$ is given by
	\[ \varphi^{\gamma}(z) = \frac{\rho^2z^2 + \rho(\lambda_1 + \lambda_2 - 2)z - (\lambda_1 + \lambda_2 - 2)}{\rho^2\lambda_2 z - \rho(\lambda_2 - 1)}. \]
Since $\lambda_1,\lambda_2 \in \O$ by assumption, all of the coefficients of $\varphi^{\gamma}$ are integers. However, as we will now see, all of the coefficients of $\varphi^\gamma$ lie in $\m$, so we need a normalized representation of $\varphi^{\gamma}$.

Since $|\lambda_3| = 1$, it follows from \eqref{eq:lambda_3} that
	\[ |\lambda_1 + \lambda_2 - 2| = |\lambda_1\lambda_2 - 1| = |\rho^2| = r^2. \]
We also claim that $|\lambda_2 - 1| \le r$. Indeed, suppose to the contrary that $|\lambda_2 - 1| > r$. Then
	\[ |(\lambda_2 - 1)^2| > |\lambda_1\lambda_2 - 1| = |\lambda_1 + \lambda_2 - 2|, \]
so that
	\[ |(\lambda_2 - 1)^2| > |\lambda_1 + \lambda_2 - 2| = |\lambda_1\lambda_2 + \lambda_2^2 - 2\lambda_2| = |(\lambda_2 - 1)^2 + (\lambda_1\lambda_2 - 1)| = |(\lambda_2 - 1)^2|, \]
a contradiction.

We now see that the absolute values of the coefficients of $\varphi^{\gamma}$ are as follows:
	\begin{align*}
		|\rho^2| = |-(\lambda_1 + \lambda_2 - 2)| = |\rho^2\lambda_2| &= r^2; \\
		|\rho(\lambda_1 + \lambda_2 - 2)| &= r^3; \\
		|\rho(\lambda_2 - 1)| &\le r^2;
	\end{align*}
so the maximum among the absolute values of the coefficients is $r^2$. We therefore divide all coefficients by $\rho^2$ to obtain the normalized representation
	\[ \varphi^{\gamma}(z) = \frac{z^2 + \frac{\lambda_1 + \lambda_2 - 2}{\rho} z - \frac{\lambda_1 + \lambda_2 - 2}{\rho^2}}{\lambda_2 z - \frac{\lambda_2 - 1}{\rho}}. \]
The resultant of the natural lift $\Phi^{\gamma}$ is
	\[ \Res(\Phi^{\gamma}) \ = \  -\frac{\lambda_1\lambda_2 - 1}{\rho^2}\  = \ -1 \ , \]
and therefore $\varphi^{\gamma}$ has good reduction. It follows that $\gamma(\zetaG) = \zeta_{D(-1,r)}$ 
is a repelling fixed point  for $\varphi$, and hence $\xi = \zeta_{D(-1,r)}$, as claimed.

To prove (B), take $\gamma(z) := \lambda_1 z$, so that $\gamma(\zetaG) = \zeta_{D(0,|\lambda_1|)}$. Then
	\[ \varphi^{\gamma}(z) \ = \ \frac{z^2 + z}{\lambda_2 z + 1/\lambda_1}\ . \]
Since $|\lambda_1| > 1 \ge |\lambda_2|$, the obvious lift $\Phi^\gamma$ is normalized,  
and we can reduce modulo $\m$:
	\[ \tilde{\varphi^{\gamma}}(z) \ = \ \frac{z^2 + z}{\tilde{\lambda_2} z} \ = \ \frac{\tilde{1}}{\tilde{\lambda_2}} (z + \tilde{1}) \ . \]
	
If $\tilde{\lambda_2} \not \in \{\tilde{0},\tilde{1}\}$, then $\tilde{\varphi^{\gamma}}$ is conjugate to the map 
$(\tilde{1}/\tilde{\lambda_2})z$ via $z \mapsto z + \tilde{1}/(\tilde{\lambda_2} - \tilde{1})$, and therefore $\varphi^{\gamma}$ has 
multiplicative reduction. Moreover, observe that
	\[ |\alpha_3| = \frac{|\lambda_1 - 1|}{|\lambda_2 - 1|} = |\lambda_1|, \]
so $\zeta_{D(0,|\lambda_1|)} = \zeta_{D(0,|\alpha_3|)}$ is the branch point of $\GammaF$.
Thus $\zeta_{D(0,|\lambda_1|)}$ satisfies (W2), proving (B)(i).

If $\tilde{\lambda_2} = \tilde{1}$, then $\tilde{\varphi^{\gamma}} = z + \tilde{1}$, 
which shows that $\varphi^{\gamma}$ has additive reduction. Hence $\zeta_{D(0,|\lambda_1|)}$ satisfies (W3), 
proving (B)(ii).

Finally, if $\tilde{\lambda_2} = \tilde{0}$, then $\tilde{\varphi^{\gamma}}$ is the constant map $\tilde{\infty}$. 
This means that $\zeta_{D(0,|\lambda_1|)}$ is not a fixed point under $\varphi$. 
Arguing just as we did for (B)(i), we see that $\zeta_{D(0,|\lambda_1|)}$ is the branch point for $\GammaF$, 
so $\zeta_{D(0,|\lambda_1|)}$ satisfies (W4), completing the proof.
\end{proof}

\section{Applications to Moduli Space}\label{sect:modulispace}

\subsection{Proof of the Main Theorem}\label{sec:MainThmProof}

We are now ready to prove our main result, Theorem~\ref{thm:main}, which says that for quadratic maps the crucial set  
gives a stratification of $\M_2(K)$ compatible with specialization to $\overline{\M}_2(k)$.
Recall that $\M_2 \cong \AA^2$ and $\overline{\M}_2 \cong \PP^2$ as schemes over $\Z$;  
by abuse of notation we view the isomorphism  
$\textbf{s} = (\sigma_1,\sigma_2) : \M_2(K) \rightarrow \AA^2(K)$ 
as an embedding $$\textbf{s} = [\sigma_1: \sigma_2: 1]:\M_2(K) \hookrightarrow \mathbb{P}^2(K) \ .$$
Finally, for a point $P \in \P^2(K)$, we denote by $\widetilde{P} \in \P^2(k)$ the specialization of $P$ modulo $\m$.

\begin{main_thm} 
Let $K$ be a complete, algebraically closed non-Archimedean field.  
Let $\varphi$ be a degree two rational map over $K$, and let $\xi$ denote the unique point in the crucial set of $\varphi$.
Then
	\begin{enumerate}
		\item $\tilde{\s([\varphi])} \in \A^2(k)$ if and only if $\xi$ satisfies {\rm(W1)}.
		\item $\tilde{\s([\varphi])} = [\tilde{1} : \tilde{x} : \tilde{0}]$ 
                for some $\tilde{x} \in k$ with $\tilde{x} \ne \tilde{2}$ if and only if $\xi$ satisfies 
               {\rm(W2)}. In this case, $\tilde{x} = \tilde{\lambda} + \tilde{\lambda}^{-1}$, 
                   where $\lambda$ is the multiplier of an indifferent fixed point for $\varphi$.
		\item $\tilde{\s([\varphi])} = [\tilde{1} : \tilde{2} : \tilde{0}]$ if and only if $\xi$ satisfies {\rm(W3)}.
		\item $\tilde{\s([\varphi])} = [\tilde{0} : \tilde{1} : \tilde{0}]$ if and only if $\xi$ satisfies {\rm(W4)}.
	\end{enumerate}
\end{main_thm}

\begin{proof}[Proof of Theorem~\ref{thm:main}]
Because $\xi$ must satisfy exactly one of (W1) -- (W4), and since $\P^2(k)$ is equal to the disjoint union
	\[ \P^2(k) \ = \ \A^2(k) 
           \sqcup \left\{[\tilde{1} : \tilde{x} : \tilde{0}] \ | \ \tilde{x} \in k, \tilde{x} \ne \tilde{2} \right\} 
           \sqcup \{[\tilde{1} : \tilde{2} : \tilde{0}] \} 
           \sqcup \{[\tilde{0}: \tilde{1} : \tilde{0}]\}, \]
it suffices to prove only the forward implications of the statements in the theorem.

First, suppose $\xi$ satisfies (W1). If $\varphi$ has three distinct fixed points, 
it follows from Proposition~\ref{prop:main_prop} that $\varphi$ has no repelling fixed points. 
Thus all of the multipliers of $\varphi$ lie in $\O$. In particular, this means that $\sigma_1, \sigma_2 \in \O$, 
so
	\[ \tilde{\s([\varphi])} = [\tilde{\sigma_1} : \tilde{\sigma_2} : \tilde{1}] \in \A^2(k). \]
If, on the other hand, $\varphi$ has a multiple fixed point, then for a suitable ordering of the multipliers 
we have $\lambda_1 = \lambda_2 = 1$, and by Proposition~\ref{prop:ramifiedfixedpt} we have $|\lambda_3|\leq 1$. 
Once again, all the multipliers of $\varphi$ lie in $\mathcal{O}$, and so
	\[ \tilde{\s([\varphi])} = [\tilde{\sigma_1} : \tilde{\sigma_2} : \tilde{1}] \in \A^2(k). \]

For $\xi$ to satisfy (W2), (W3), or (W4), the map $\varphi$ must have at least one repelling classical fixed point 
and one non-repelling classical fixed point. Indeed, in the case that $\varphi$ has a multiple fixed point, 
this follows from Proposition~\ref{prop:ramifiedfixedpt}; in the case that $\varphi$ has three distinct fixed points, 
we know from Proposition~\ref{prop:main_prop} that $\varphi$ must have at least one classical repelling fixed point, 
in which case $\varphi$ also has a non-repelling fixed point by Lemma~\ref{lem:repelling}. 
Therefore, we may assume for the remainder of the proof that $|\lambda_1| > 1 \ge |\lambda_2|$.\\

\medskip
Suppose $\xi$ satisfies (W2). It follows from Proposition~\ref{prop:ramifiedfixedpt} that $\varphi$ 
must have three distinct fixed points, and from Proposition~\ref{prop:main_prop} 
we must have $\tilde{\lambda_2} \not \in \{\tilde{0},\tilde{1}\}$. Using the explicit formula for $\lambda_3$ 
in Lemma~\ref{lem:lambda_3}, together with the fact that $|\lambda_2| = 1$, we find that $|\lambda_3| = 1$. 
Therefore $(1/\lambda_1)\sigma_1$ and $(1/\lambda_1)\sigma_2$ lie in $\mathcal{O}$; 
reducing modulo $\m$ yields
	\begin{align*}
	\tilde{\left(\frac{\sigma_1}{\lambda_1}\right)} \ &= \ \tilde{1} + \tilde{\left(\frac{\lambda_2}{\lambda_1}\right)} + \tilde{\left(\frac{\lambda_3}{\lambda_1}\right)} \ = \ \tilde{1}\ ; \\
	\tilde{\left(\frac{\sigma_2}{\lambda_1}\right)} \ &= \ \tilde{\lambda_2} + \tilde{\lambda_3} + \tilde{\left(\frac{\lambda_2\lambda_3}{\lambda_1}\right)} \ = \ \tilde{\lambda_2} + \tilde{\lambda_3}\ .
	\end{align*}
Thus 
\[\widetilde{\textbf{s}([\varphi])} \ = \ [\tilde{1}: \tilde{\lambda_2} + \tilde{\lambda_3} : \tilde{0}]\ .\] 
Once again using the formula for $\lambda_3$ from Lemma~\ref{lem:lambda_3}, we have
	\[ \lambda_3 \ = \ \frac{1 + \lambda_2/\lambda_1 - 2/\lambda_1}{\lambda_2 - 1/\lambda_1}\ . \]
Reducing modulo $\m$, we therefore have $\tilde{\lambda_3} = \tilde{\lambda_2}^{-1}$. 
Letting $\tilde{x} = \tilde{\lambda_2} + \tilde{\lambda_2}^{-1}$ and noting that $\tilde{\lambda_2} \ne \tilde{1}$ 
implies $\tilde{x} \ne \tilde{2}$ completes the proof of (B). 

\medskip
Now suppose that $\xi$ satisfies (W3). In the case that $\varphi$ has a multiple fixed point, 
it follows from Proposition~\ref{prop:ramifiedfixedpt} and our assumption 
$|\lambda_1| > 1$ that $\lambda_2 = \lambda_3 = 1$. Thus $\widetilde{(\sigma_1/\lambda_1)} = \tilde{1}$ 
and $\widetilde{(\sigma_2/\lambda_1)} = \tilde{2}$. Similarly, if $\varphi$ has three distinct 
fixed points, then by Proposition~\ref{prop:main_prop} and Lemma~\ref{lem:lambda_3} we have $\tilde{\lambda_2} =\tilde{1}$ 
and $\tilde{\lambda_3} = \tilde{1/\lambda_2} = \tilde{1}$.
Here again we have $\widetilde{(\sigma_1/\lambda_1)} = \tilde{1}$ and $\widetilde{(\sigma_2/\lambda_1)} = \tilde{2}$.. 
Thus in every case, we have 

\[\widetilde{\textbf{s}([\varphi])} \ = \ [\tilde{1}: \tilde{2} : \tilde{0}]\] as asserted.

\medskip
Finally suppose that $\xi$ satisfies (W4). It follows from Proposition~\ref{prop:ramifiedfixedpt} that $\varphi$ 
must have three distinct fixed points, and from Proposition~\ref{prop:main_prop} we must have $\tilde{\lambda_2} = 0$. 
Since $\tilde{\lambda_2} = 0$, we have $|\lambda_2| < 1$, and therefore $|\lambda_3| > 1$ by Lemma~\ref{lem:repelling}. 
We now observe that
	\begin{align*}
		\left|\frac{\sigma_1}{\lambda_1\lambda_3}\right| = \left|\frac{1}{\lambda_3} + \frac{\lambda_2}{\lambda_1\lambda_3} + \frac{1}{\lambda_1}\right| &< 1;\\
		\left|\frac{\sigma_2}{\lambda_1\lambda_3}\right| = \left|\frac{\lambda_2}{\lambda_3} + 1 + \frac{\lambda_2}{\lambda_1}\right| &= 1.
	\end{align*}
Therefore
	\[ \tilde{\textbf{s}([\varphi])} = \left[ \tilde{\left(\frac{\sigma_1}{\lambda_1\lambda_3}\right)} : \tilde{\left(\frac{\sigma_2}{\lambda_1\lambda_2}\right)} : \tilde{\left(\frac{1}{\lambda_1\lambda_3}\right)}\right] = [\tilde{0} :\tilde{1} : \tilde{0}], \]
as claimed.
\end{proof}

\subsection{An Application to Repelling Periodic Points}\label{sec:hsia}

We now use the main theorem to prove a special case of a conjecture of Hsia. 
For a rational map $\varphi \in K(z)$, 
let $\mathcal{J}_\varphi(K)$ denote the (classical) Julia set of $\varphi$, 
let $\mathcal{R}_\varphi(K)$ denote the set of all classical repelling periodic points for $\varphi$, 
and let $\overline{\mathcal{R}_\varphi(K)}$ be its closure in $\PP^1(K)$. 

It is known over the complex numbers that $\mathcal{J}_\phi(\CC) = \overline{\mathcal{R}_\phi(\CC)}$; the analagous result is not known when $K$ is non-Archimedean, though it is conjectured to be true:

\begin{conj}[{Hsia, \cite[Conj. 4.3]{hsia:2000}}]
Let $\varphi$ be a rational function defined over a non-archimedean field with $\deg \varphi \ge 2$. Then $\mathcal{J}_\varphi(K) = \overline{\mathcal{R}_\varphi(K)}$.
\end{conj}

Using Theorem \ref{thm:main}, we show that Hsia's conjecture holds for a quadratic rational map  
over a complete, algebraically closed non-archimedean field. 

\begin{main_prop}
Let $\varphi$ be a quadratic rational map defined over $K$. 
Then $\mathcal{J}_\varphi(K) = \overline{\mathcal{R}_\varphi(K)}$.\end{main_prop}

\begin{proof}
We separate the proof based on whether $\varphi$ has potential good reduction or bad reduction. 

If $\varphi$ has potential good reduction, the Berkovich Julia set is a single point in $\hberk$ (see \cite{FRL} Proposition 0.1); 
thus $\varphi$ has no type I repelling periodic points (such would necessarily be Julia), 
hence $\mathcal{J}_\varphi(K)= \emptyset = \overline{\mathcal{R}_\varphi(K)}$ as desired.

If $\varphi$ has bad reduction, then by Theorem~\ref{thm:main} the image
 $\textbf{s}([\varphi]) \in \AA^2(K) \subset \PP^2(K)$ 
cannot specialize to $\mathbb{A}^2(k)$. 
It follows that $\varphi$ must have a type I repelling fixed point, for if all of its type I fixed points 
were non-repelling then the symmetric functions in their multipliers would lie in $\mathcal{O}_K$. 
By a theorem of B\'ezevin (\cite{Bez}, Th\'eor\`eme 3), 
this implies that $\mathcal{J}_\varphi(K) = \overline{\mathcal{R}_\varphi(K)}$. 
\end{proof}

\end{document}